\documentclass[]{article}

\usepackage{amssymb} 

\usepackage{amsmath} 

\usepackage{latexsym}

\usepackage{graphicx}

\usepackage{tikz}

\usetikzlibrary{calc}

\usepackage{pgf,tikz}

\usetikzlibrary{arrows}

\usepackage{amscd}

\usepackage{relsize}

\usepackage{tikz-cd}

\usepackage{ color, amsmath, amssymb, amsfonts, amscd, graphicx,latexsym, }

\usepackage{caption, subcaption}

\usepackage[pdf]{pstricks}

\usepackage[affil-it]{authblk}

\usepackage{amsthm}

\newtheorem{theorem}{Theorem}

\newtheorem{remark}{Remark}

\newtheorem*{udefinition}{Definition}

\newcommand{\R}{\mathbb R}
\newcommand{\Z}{\mathbb Z}
\newcommand{\C}{\mathbb C}

\title{Straightening Billiard Trajectories in Flat Disks and Katok-Zemlyakov Construction}

\author{ \.{I}smail Sa\u{g}lam \thanks{Electronic address: \texttt{isaglamtrfr@gmail.com}  }}

\affil{Adana Alparslan Turkes Science and Technology University  }

\date{}

\begin{document}

\maketitle

\begin{abstract}
We provide a method to straighten each billiard trajectory in a flat disk. As an application, we show that for each point on the disk and for almost all directions, the closure of the corresponding  billiard trajectory contains a vertex. We generalize Katok-Zemlyakov construction to the flat disks with rational angle data: for each rational flat disk  we obtain a translation surface and a projection to the doubling of the disk. We calculate Euler characteristics of this translation surface.

\end{abstract}


\section{Introduction}
\label{introduction}
Billiard flows in polygons is a subject which is being studied extensively.  Ergodicity and minimality of the flows, number of closed geodesics and generalized diagonals are among the most important topics in this area. See the surveys \cite{Gutkin1},\cite{Gutkinsurvey},\cite{rational-billiard}.

Polygons are special examples of the flat surfaces.  A flat surface is a surface together with a metric obtained by gluing Euclidean triangles along their edges by isometries. See \cite{ISP},\cite{MT},\cite{MT1},\cite{MT2},\cite{MT3} for more information about flat surfaces. For each point $p$ in the surface, there exists a well defined notion of angle, $\theta_p$. If an interior point has angle $2\pi$, then it is called non-singular. Otherwise, it is called singular. If a boundary point has angle $\pi$, then it is called non-singular. Otherwise, it is called singular. Let $S$ be a compact flat surface. Let $i(S)$ and $b(S)$ be the boundary and the interior of $S$, respectively. The following formula (Gauss-Bonnet) holds:

\begin{align*}
\sum_{p\in i(S)}(2\pi -\theta_p)+\sum_{p \in b(S)}(\pi-\theta_p)=2\pi \chi(S),
\end{align*}
where $\chi(S)$ is the Euler characteristics of $S$. 

It is hard to say that the study of the billiard flow on any compact flat surface is possible. Nevertheless, there is a family of flat surfaces that has been studied extensively up to now. This family consists of the closed flat surfaces with trivial holonomy groups. Such a surface is called a translation surface. See \cite{Zorich}, \cite{Masurer} for more information about the translation surfaces. The aim of this manuscript is to point out that there is another family of the flat surfaces for which it is not too much to hope that the billiard flow can be studied in a systematic manner. This family consists of flat disks.

A flat disk is a closed, oriented topological disk together with a flat metric on it. Polygons are simplest examples of the flat disks. Each flat disk $D$ may have finitely many singular interior and singular boundary points. Gauss-Bonnet formula becomes

\begin{align*}
\sum_{p\in i(D)}(2\pi -\theta_p)+\sum_{p \in b(D)}(\pi-\theta_p)=2\pi.
\end{align*}

\begin{figure}

	\centering
	\includegraphics[scale=0.6]{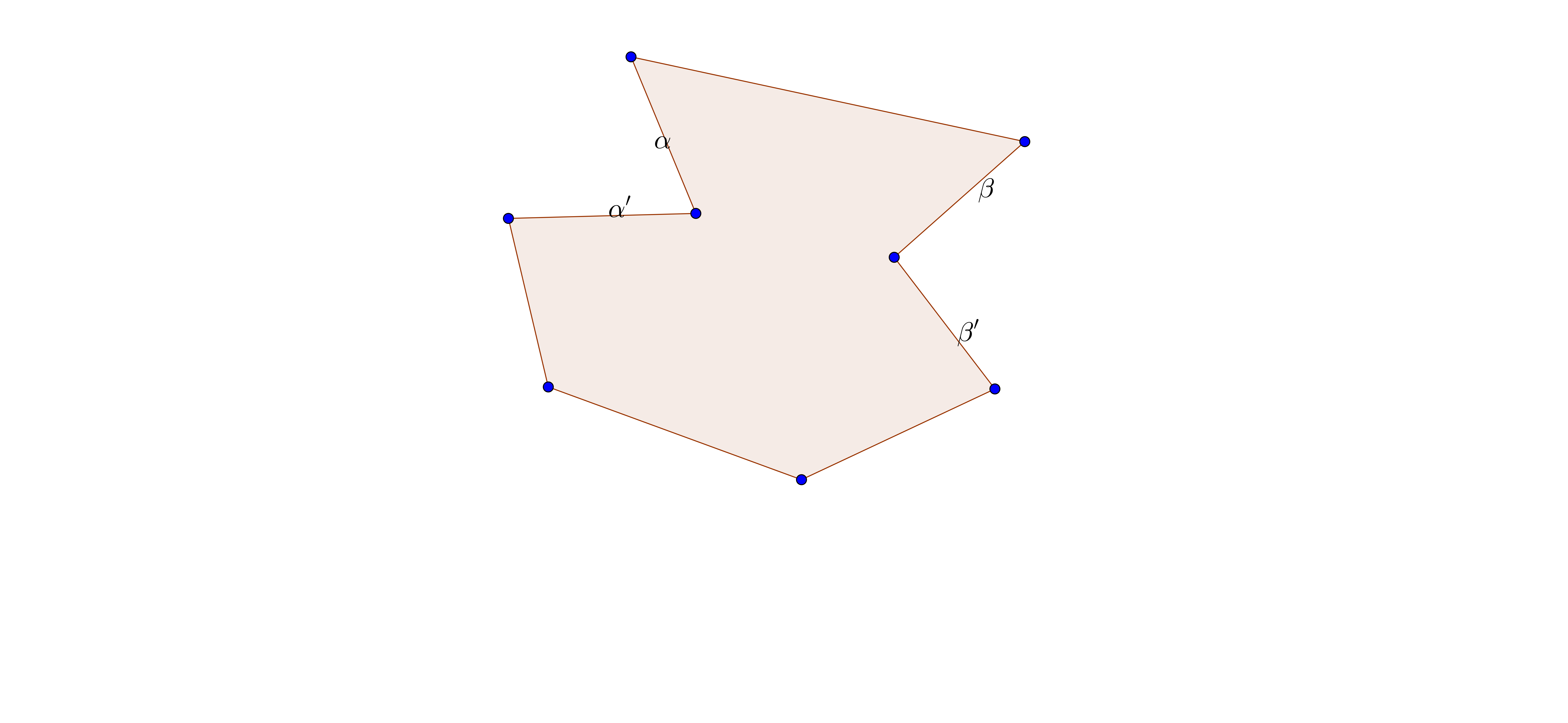}
	\caption{Constructing a flat disk out of a polygon}
		\label{flatdisk}
\end{figure}
 
 Now we construct a flat disk. See Figure \ref{flatdisk}.  Note that the length of the edge $\alpha$ is equal to the length of the edge $\alpha'$, and the length of the edge $\beta$ is equal to the length of the edge $\beta'$. If we glue $\alpha$ and $\alpha'$, $\beta$ and $\beta'$, then we get a flat disk with two singular interior points and four singular boundary points. Indeed, any flat metric on a disk may be constructed from a flat disk in a similar manner. See Section \ref{unfoldingflatdisks}.
 
 There is a simple but quite useful trick used in the study of the polygonal billiards to unfold a billiard trajectory in a polygon. See Section \ref{unfoldingpolygons}. The idea of the trick is to reflect the billiard table instead of reflecting the billiard ball. We give a generalization of this trick for an arbitrary flat disk. See Section \ref{unfoldingflatdisks}. As a result, we generalize a theorem obtained by Boldrighini, Keane and Marchetti \cite{boldrighini}. See also \cite{sinai}
 
 In \cite{boldrighini}, it was proved that for any point in a polygon and for almost all directions, the closure of the billiard trajectory passing through this point having the same direction contains a vertex. See \cite{trou} for an improvement of this result. We show that for all points on a flat disk and almost all directions, the closure of the billiard trajectory contains a singular boundary point or a singular interior point. See Section \ref{theorembkm}.
 
 A polygon is rational if its angles are rational multiples of $\pi$. From a rational polygon one can obtain a translation surface. This is known as Katok-Zemlyakov construction.  See \cite{top-tran}. This surface is called invariant surface. By that way one can argue that a rational polygon contains closed billiard trajectories, and give upper and lower bounds for number of closed trajectories and generalized diagonals. See \cite{masurbound}. In this manuscript we do a parallel construction for rational flat disks. 
 
 A flat disk is called rational if the angles at its singular points are rational multiples of $\pi$. This is equivalent that the holonomy group of the doubling of the flat disk is finite,  where the the doubling of a flat surface with boundary is the surface obtained by gluing two copies of the surface along their boundaries. For such a flat disk, we can construct a translation surface and a projection from the translation surface to the disk or its doubling.  We show that the translation surface that we obtain is unique; it does not depend on how we cut the disk to obtain this translation surface. See Sections \ref{KZflatdisks}, \ref{uniqueness}.  Finally, we calculate the Euler characteristic of this translation surface. See Section \ref{eulerchar}.

 \section{Straightening Billiard Trajectories and a Generalization of the Theorem of Boldrighini, Keane and Marchetti}
 \label{section2}
 In this section we recall how to unfold billiard trajectories in the polygons. Then we extend the construction for the flat disks. Also, we extend the theorem of Boldrighini, Keane and Marchetti to the flat disks \cite{boldrighini}. 
 
 \subsection{Unfolding Trajectories in Polygons}
 \label{unfoldingpolygons}
 Let $P$ be a polygon in $\R^2$. Label its edges as $e_1,\dots,e_n$. Let $\gamma$ be a billiard trajectory which never hits the vertices $y_1,\dots, y_n$. Assume that $\gamma$
 first hits the edge $\omega_1$, then the edge $\omega_2$, and in general let $\omega_k$ be the edge which $\gamma$ hits at its $k$-th intersection with the boundary $b(P)$. Therefore we have an infinite word 
 
 \begin{align*}
 \omega=(\omega_1,\omega_2,\dots), \ \omega_k \in \{e_1, \dots,e_n\}.
 \end{align*}
Let $P_0=P$ and $P_1$ be the polygon obtained by reflecting $P_0$ along $\omega_1$. In general, let $P_k$, $k\geq 1$, be the copy of $P$ obtained by reflecting $P_{k-1}$ along its edge  $\omega_n$. See the Figure \ref{unfoldingtriangle}.
\begin{figure}
	\hspace*{-10cm}                                                           
	
	\centering
	\includegraphics[scale=0.6]{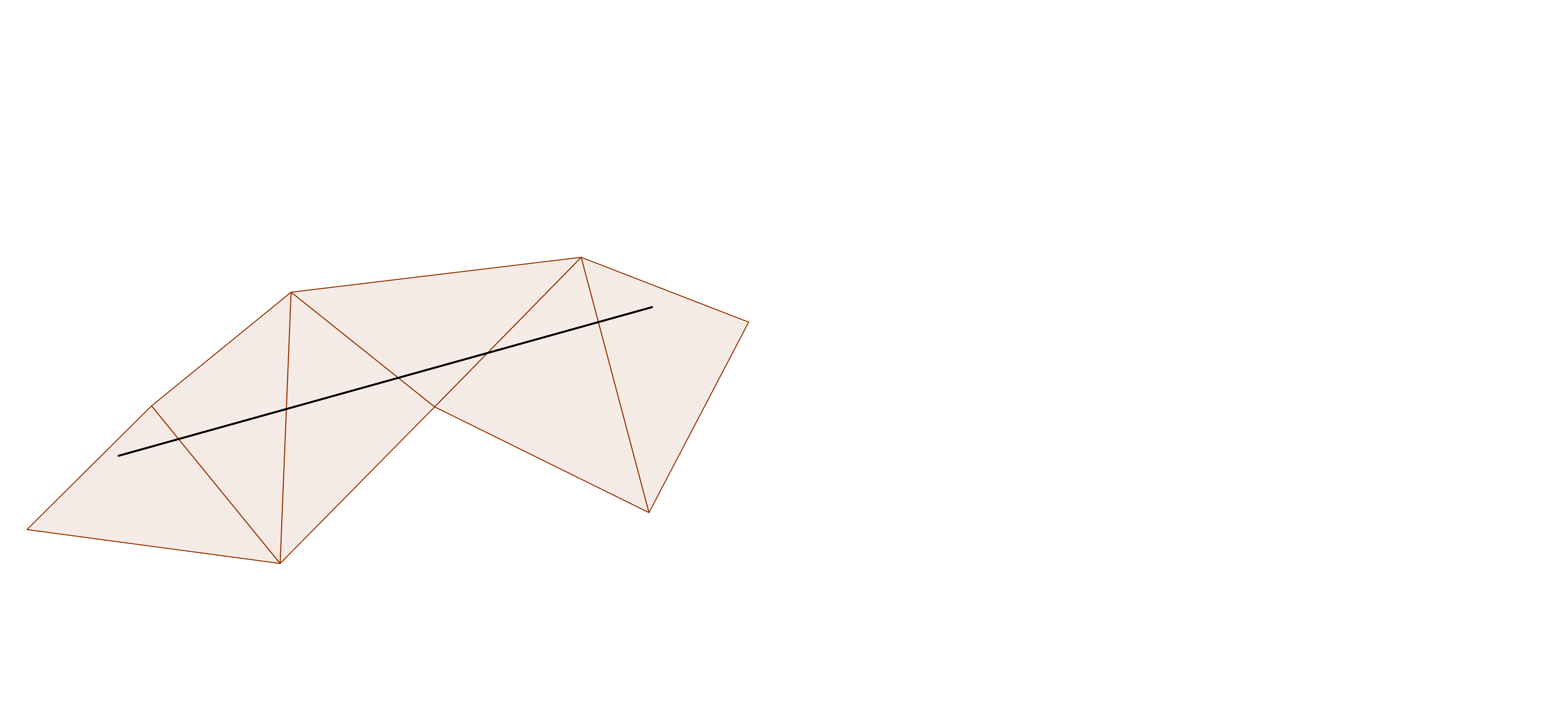}
	\caption{Unfolding billiard trajectory in a triangle}
	\label{unfoldingtriangle}
\end{figure}
Let $Q_0=P_0$ and $Q_1=P_0\sqcup_{\sim}P_1$, where we identify the edges labeled by $\omega_1$ in $P_0$ and $P_1$. Then we can unfold $\gamma$ at $\omega_1$ as a straight line segment. By that way, we can obtain a line segment in $Q_1$ which hits the edge $\omega_2$. Call this segment $\gamma_1$. 

In general, we can form the polygon $Q_k=Q_{k-1}\sqcup_{\sim}P_k$, where $Q_{k-1}=P_0\sqcup_{\sim}P_1\dots\sqcup_{\sim}P_{k-1}$ and $P_k$ are glued through the edges of $P_{k-1}$ and $P_k$ labeled by $\omega_k$. Similarly, I can unfold $\gamma_{k-1}$ at $\omega_k$ as a straight line segment intersecting $\omega_{k+1}$. We call this line segment $\gamma_k$. This means that I can obtain an unbounded complete flat surface 

$Q_{\infty}=P_0\sqcup_{\sim}P_1\sqcup_{\sim}P_{2}\dots$

\noindent with a non-singular interior and a straight half-line $\gamma_{\infty}$. Also, there is a projection $\Pi_{P}: Q_{\infty} \rightarrow Q$ sending $\gamma_{\infty}$
 to $\gamma$. It is clear that the restriction of this projection to each $P_i \subset Q_{\infty}$ is an isometry between $P_i$ and $P$.
 
 \subsection{Unfolding of Trajectories on Flat Disks}
 \label{unfoldingflatdisks}
 Before we proceed, let us make it clear what we mean by {\it unfolding}. In the case of polygons, we obtained a flat surface $Q_{\infty}$ with non-singular interior and a line $\gamma_{\infty}$ on it such that there is a projection $\Pi_P: Q_{\infty}\rightarrow P$ sending $\gamma_{\infty}$ to the billiard trajectory $\gamma$. For the case of the flat disks, we need to seek for such a surface and a projection map. 
 
 Given a flat disk $D$,  one my try to construct $Q_{\infty}$ as follows. Let $e_1, e_2, \dots, e_n$ be the edges of $D$. Let $\gamma$ be a billiard trajectory which never hits the singular points. Let

  \begin{align*}
 \omega=(\omega_1,\omega_2,\dots), \ \omega_k \in \{e_1, \dots,e_n\}.
 \end{align*}
 
 \noindent where $\omega_k$ is the edge on which the trajectory intersects the boundary at the $k$-th time. Let $D_0, D_1, \dots$ be the copies of $D$. Let $Q_0=D_0$ and $Q_1=D_0\sqcup_{\sim}D_1$, where we identify the edges labeled $\omega_1$ in $D_0$ and $D_1$.  In general, let
 
 $$Q_{k}=Q_{k-1}\sqcup D_k,$$
 where we glue the edgs labeled $\omega_k$ in $Q_{k-1}$ and $D_k$. Similarly
 
 $$Q_{\infty}=D_1\sqcup_{\sim}D_2\sqcup_{\sim}D_3\dots$$ 
 
 Indeed, $Q_{\infty}$ is not good for our purposes. It has infinitely many singular points in its interior if interior of $D$ is singular.
 
 Another important thing is that there may be trajectories in $D$ which never hit the boundary and the singular points. See the Figure \ref{periodic}. Take two copies of a square and glue the edges $e_1$ with $e_1'$, $e_2$ with $e_2'$, $e_3$ with $e_3'$ to get a flat disk with two singular points on its interior. The geodesic shown in the figure is periodic and never hits the boundary. Note that this situation does not arise in the case where the angle at each interior point is greater than or equal to $2\pi$. In that case, a trajectory either hits the boundary or the singular points. This can be proved by a simple argument involving Gauss-Bonnet Formula.

 \begin{figure}
 	\hspace*{-10cm}                                                           
 	
 	\centering
 	\includegraphics[scale=0.6]{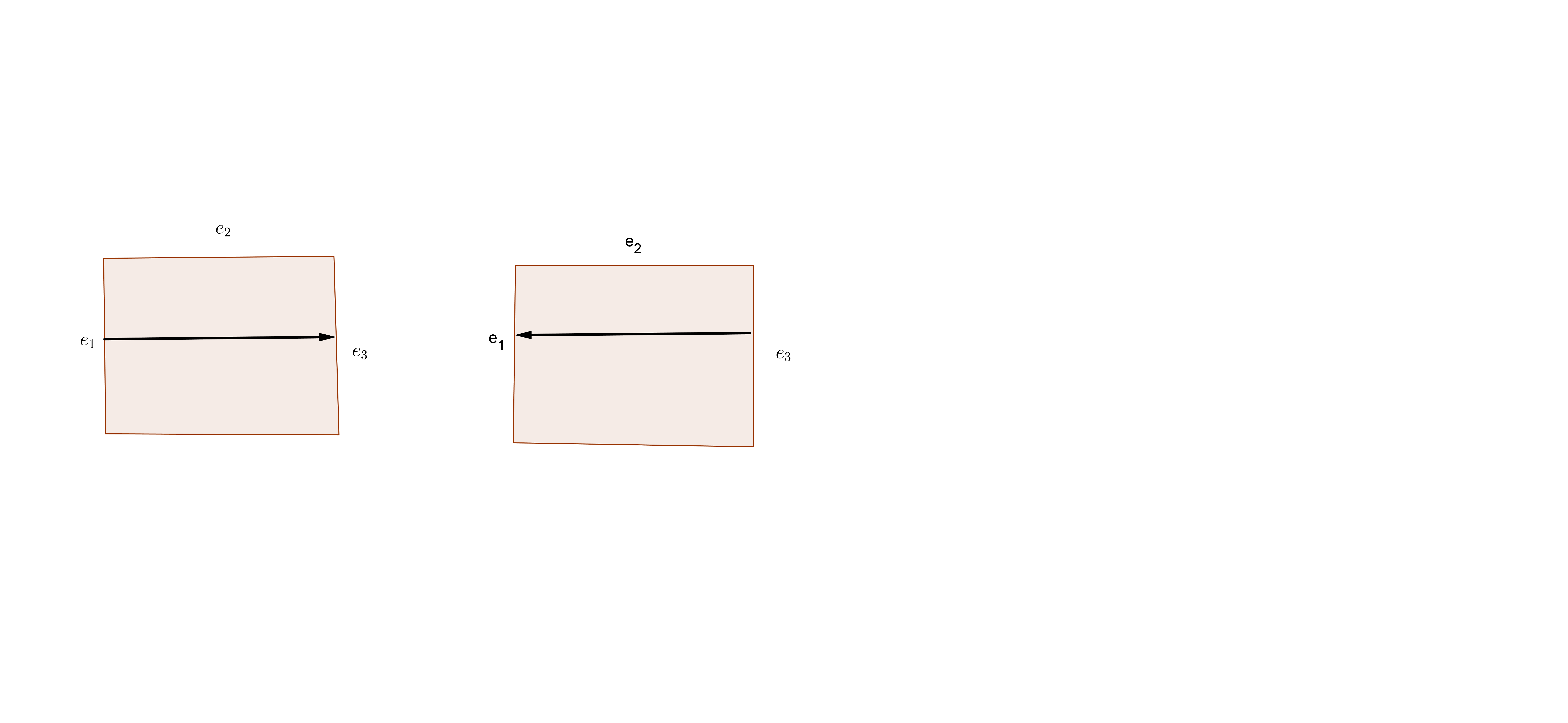}
 	\caption{A periodic billiard trajectory which never hits the boundary}
 	\label{periodic}
 \end{figure}

 Now we explain how to unfold the billiard trajectories in the flat disks. Let $D$ be a flat disk with $m$ singular interior points $x_1,\dots,x_m$ and $n$ singular boundary points $y_1,\dots,y_n$. Let $e_1, \dots, e_n$ be the boundary segments or the edges of $D$.

 Let $a_1, \dots,a_m$ be  $m$ not self-intersecting broken geodesics such that 
 
 \begin{enumerate}
 	\item 
 	$a_i$ joins $x_i$ and a singular boundary point $y_{j(i)}$,
 	\item
 	$a_i\cap b(D)=\{y_{j(i)}\}$,
 	\item
 	$a_i\cap a_j \subset b(D)$.
 	
 \end{enumerate}
 
Let $\gamma$ be a billiard trajectory which never hits the singular points. $\gamma$ will intersect the set 
\begin{align}
\frak{a}=a_1\cup a_2\dots\cup a_m \cup e_1 \cup \dots \cup e_n
\end{align}
infinitely many times.  Also, since $D$ is oriented it makes sense to say $\gamma$ hits $a_i$ from left or right. Let 

$$\omega=(\omega_1,\omega_2,\dots)$$
 be the infinite word consisting of the letters
 
 $$a_1^l,a_1^r,\dots, a_m^l,a_m^r, e_1,\dots e_n$$ 
 
\noindent  such that 
 
 \begin{itemize}
 	\item 
 	$\omega_k=a_i^l$ if the trajectory hits $\frak{a}$ from the left $a_i$ at the $k$-th time,
 	
 	\item
 	$\omega_k=a_i^r$ if the trajectory hits  $\frak{a}$ from the right of $a_i$ at the $k$-th time,
 	
 	\item
 	$\omega_k=e_i$ if the trajectory hits $\frak{a}$ at $e_i$
 at the $k$-th time.

 \end{itemize}

Cut $D$ through $a_i$'s to get a flat disk $D^*$ with a non-singular interior and $2m+n$ boundary components. Clearly, we can label boundary components of $D^*$ with the symbols

 $$a_1^l,a_1^r,\dots, a_m^l,a_m^r, e_1,\dots e_n$$

\noindent so that this labeling coincides with the labeling of $D$.

 Take countably many copies $D_0^*,D_1^*,\dots$ of $D^*$. Let $Q_0=D_0^*$ and $Q_1=D_0^*\sqcup_{\sim}D_1^*,$ where we identify the edges
 
 \begin{enumerate}
 	\item 
 	labeled $e_i$ in $D_0^*$ and $D_1^*$ if $\omega_1=e_i$,
 	\item
 	labeled $a_i^l$ in $D_0^*$ and $a_i^r$ in $D_1^*$ if $\omega_1=a_i^l$,
 	\item
 	labeled $a_i^r$ in $D_0^*$ and $a_i^l$ in $D_1^*$ if $\omega_1=a_i^r$.
 \end{enumerate} 
 
 Observe that we can unfold $\gamma$ to a line segment $\gamma_1$ hitting the edge $\omega_2$ in $D_1^*$. We can define 
 \begin{align}
 Q_k=Q_{k-1}\sqcup_{\sim}D_k^*= D_0^*\sqcup_{\sim}D_1^*\dots\sqcup_{\sim}D_k^*
 \end{align}
\noindent as follows. We identify
the edges 
  \begin{enumerate}
 	\item 
 	labeled $e_i$ in $D_{k-1}^*$ and $D_k^*$ if $\omega_k=e_i$,
 	\item
 	labeled $a_i^l$ in $D_{k-1}^*$ and $a_i^r$ in $D_k^*$ if $\omega_k=a_i^l$,
 	\item
 	labeled $a_i^r$ in $D_{k-1}^*$ and $a_i^l$ in $D_k^*$ if $\omega_k=a_i^r$.
 \end{enumerate} 
 We can unfold $\gamma_{k-1}$ to a line segment in $Q_k$ which hits the edge $\omega_{k+1}$ in $D_k^*$. Call this segment $\gamma_k$. Clearly, there is a projection $\Pi_k: Q_k \rightarrow D$ which sends $\gamma_k$ to $\gamma$ and is an  isometry when restricted to the  the interior of each $D_i^*$, $0\leq i \leq k$. Let 
 
  \begin{align}
 Q_{\infty}= D_0^*\sqcup_{\sim}D_1^*\sqcup_{\sim}D_2^*\dots
 \end{align}

 \noindent $Q_{\infty}$ is a flat surface which is unbounded, complete and has non-singular interior. Clearly, there is a projection
 
 \begin{align}
 \Pi_{\infty}: Q_{\infty}\rightarrow D 
  \end{align}
 \noindent and a straight half-line $\gamma_{\infty}\subset Q_{\infty}$ such that $\Pi_{\infty}(\gamma_{\infty})$ is $\gamma$. Clearly, $\Pi_{\infty}$ induces an isometry between the interior of $D_k \subset Q_{\infty}$ and $D \setminus(a_1\cup a_2 \dots a_m \cup e_1\cup e_2 \dots \cup e_n)$.

 \begin{figure}
 	\hspace*{-10cm}                                                           
 	
 	\centering
 	\includegraphics[scale=0.6]{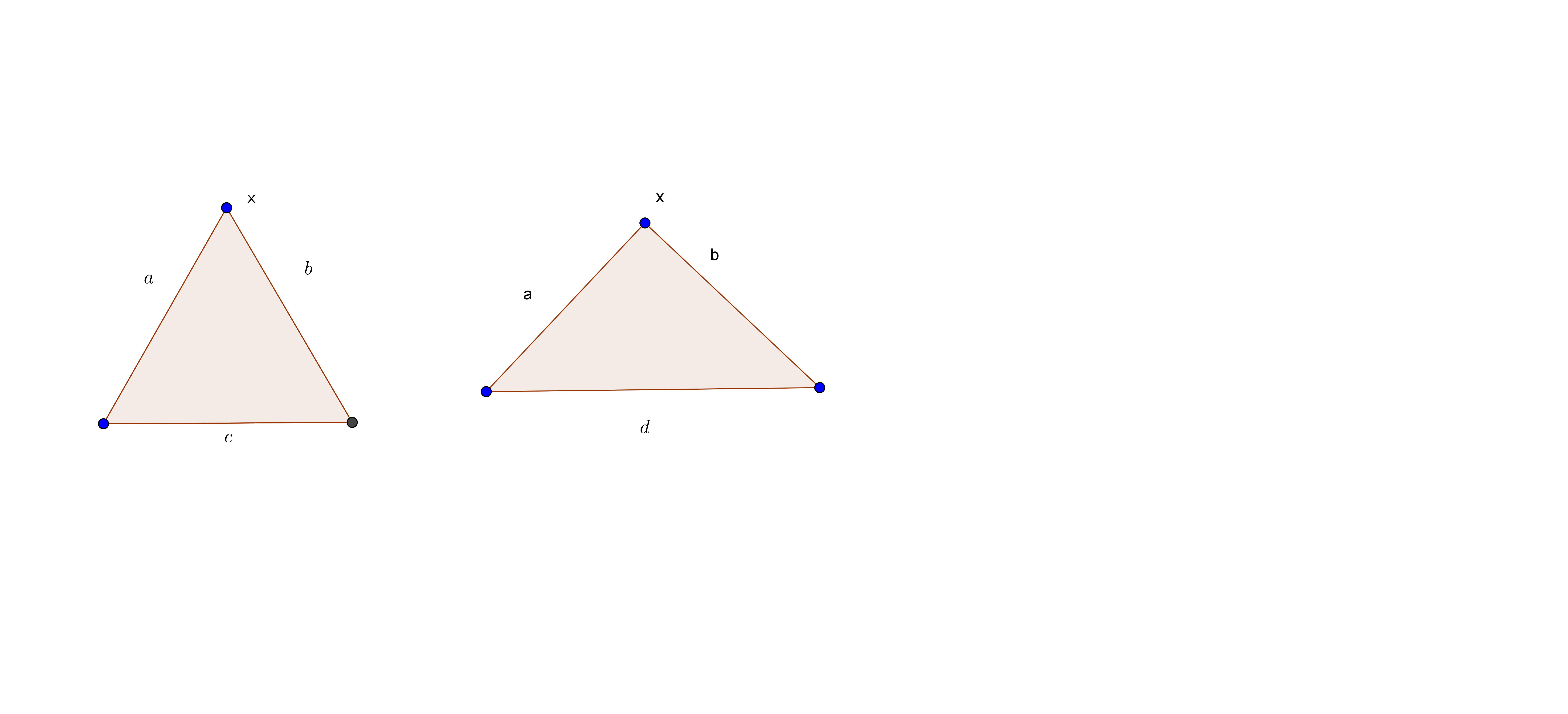}
 	\caption{A flat surface out of an equilateral and a right isosceles triangle}
 	\label{equilateral}
 \end{figure}
 
 \paragraph{Example:} Take an equilateral triangle of the edge length 1 and isosceles  right triangle having edge lengths $1,1,\sqrt{2}$. Label their edges as in the Figure \ref{equilateral}. Glue the edges having the same labels to get a flat disk with one singular interior point and two singular boundary points. Call this disk $D$. Let $D^*$ is the polygon obtained by cutting $D$ along $a$. It is shown in  Figure \ref{unfoldingflat}. Let $\gamma$ be a billiard trajectory in $D$. We can unfold it as follows. Whenever the trajectory of $\gamma$ hits $a$ (from left or right) We put a copy of $D^*$ in the plane which is obtained by rotating $D^*$ about $x$ by an angle of 150$^\circ$ . Whenever the trajectory of $\gamma$ hits $c$ (or $d$), we put a copy of $D^*$ in the plane which is obtained by reflecting $D^*$   along $c$ (or $d$). See the Figure \ref{unfoldingflat}. 
 \begin{figure}
 	\hspace*{-10cm}                                                           
 	
 	\centering
 	\includegraphics[scale=0.85]{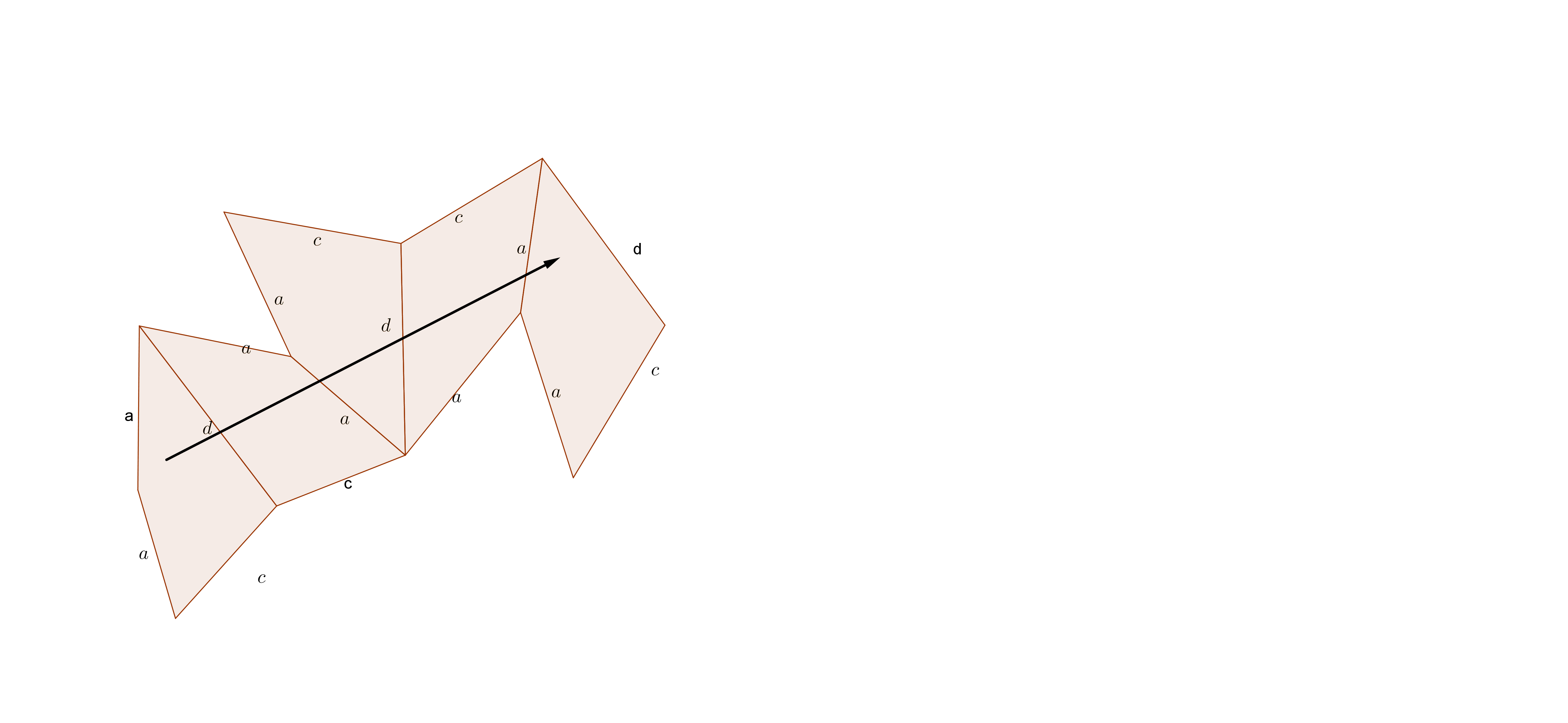}
 	\caption{Unfolding billiard trajectory in a flat disk}
 	\label{unfoldingflat}
 \end{figure}
 \begin{remark}
 	It may be possible that the boundary of $D$ is non-singular. See Figure \ref{periodic}. In that case one can choose an arbitrary point on the boundary and join singular 
 	interior points and this point by some broken geodesics to obtain $D^*$, and construct  $Q_{\infty}$ in a similar manner. Indeed, in general, the boundary points that we choose to specify $a_i$'s do not need to be singular. We choose singular boundary points since this makes the notation considerably simple. 
 \end{remark}
   
 \subsection{The theorem of Boldrighini, Keane and Marchetti}
 \label{theorembkm}
 
 In this section we prove that for all points on a flat disk and for almost all directions, the billiard trajectory passing through this point having the same  direction contains a vertex in its closure. Note that by { \it almost all directions} we mean that a set of full measure with respect to Lebesque measure on the unit tangent space of a point in the flat disk.
 
 Let $D$ be a flat disk. Let $q\in D$ be a non-singular interior point. 
 
 \begin{align}
 \mathbb{U}_q=\{v \in \mathbb{T}_q: \lvert v \rvert= 1 \}\equiv
 \R/2\pi \Z,
 \end{align} 
 \noindent where $\mathbb{T}_q$ is the tangent space at $q$. We identify $\mathbb{U}_q$ with the unit circle on $\C$ and endow it with the Lebesque measure on $\R/2\pi \Z$. Let $\rho$ be the Lebesque measure on $\mathbb{U}_q$. 
 
 \begin{theorem}
 	Let $q \in D$ be arbitrary point. For almost all $v \in \mathbb{U}_q$, the billiard trajectory initiating at $q$ and having direction $v$ contains a singular vertex in its closure. 
 \end{theorem}
 \begin{proof}
 	Let $\delta>0$ and $\theta \in \R/2\pi \Z$. Let $\gamma(\theta)$ bet the trajectory passing through $q$ and having the direction $\theta$. Let $V_0$ be the set of singular points of $D$. Let
 	\begin{align}
 	N_{\delta}=\{\theta \in \R/2\pi \Z: \text{dist}(\gamma(\theta), V_0) \geq \delta)\}
 	\end{align}
 	\noindent We will show that $\rho({N_{\delta}})=0$. It is clear that this is enough to prove the theorem. 
 	
 	Assume that there exists $\delta >0$ such that $\rho({N_{\delta}})>0$. Therefore there exists $\theta_0$ 
 	such that 
 	
 	\begin{align}
 	\lim_{\epsilon \to 0} \frac{\rho(N_{\delta}\cap[\theta_0,\theta_0+\epsilon])}
 	{\epsilon}=1.
 	\end{align}
 	\noindent In other words, $\theta_0$ is a density point of $N_{\delta}$. Let $y=(q,\theta_0)$ and $y'=(q,\theta_0+\epsilon)$. Let $\gamma$ and $\gamma'$ the trajectories having initial point $q$ and directions $\theta_0$ and $\theta_0+\epsilon$. 
 	
 	Assume that $D$ has $m$ singular interior points  $\{x_1,\dots,x_m\} \subset V_0$. Let $e_1,\dots,e_n$ be the edges of $D$. Let $a_1,\dots,a_m$ be broken not self-intersecting geodesics joining these singular points and boundary. They satisfy also the following conditions: 
 	 \begin{enumerate}
 		\item 
 		$a_i$ joins $x_i$ and a singular boundary point $y_{j(i)}$,
 		\item
 		$a_i\cap b(D)=\{y_{j(i)}\}$,
 		\item
 		$a_i\cap a_j \subset b(D)$.
 		
 	\end{enumerate}
 	
 	Let $\omega$ and $\omega'$ be the word sequences that we obtain from $\gamma$ and $\gamma'$ as in Section \ref{unfoldingflatdisks}. 
 	
 	\begin{align}
 	\omega=(\omega_1,\omega_2, \dots)\ \ \
 	 	\omega'=(\omega'_1,\omega'_2, \dots)
 	\end{align}
 	\noindent where $\omega_i,\ \omega'_i \in \{a_1^l,a_1^r,\dots,a_m^l,a_m^r,e_1,\dots,e_n\}$.  Let $D^*$ be the disk obtained from $D$ by cutting it through $a_1,\dots,a_m$. Assume that $\omega=\omega'$. Then there exists a flat surface $Q_{\infty}$ obtained from the copies of $D^*$ and two half-lines on it making an angle of $\epsilon$ with each other. Call these lines $\gamma_{\infty}$ and $\gamma'_{\infty}$. This implies that $Q_{\infty}$ contains a V-shaped region. See the Figure \ref{vshape}. But this is impossible since the diameter of $D^*$ is finite. Therefore $\omega\neq \omega'$.
 	
 	\begin{figure}
 		\hspace*{-10cm}                                                           
 		
 		\centering
 		\includegraphics[scale=0.6]{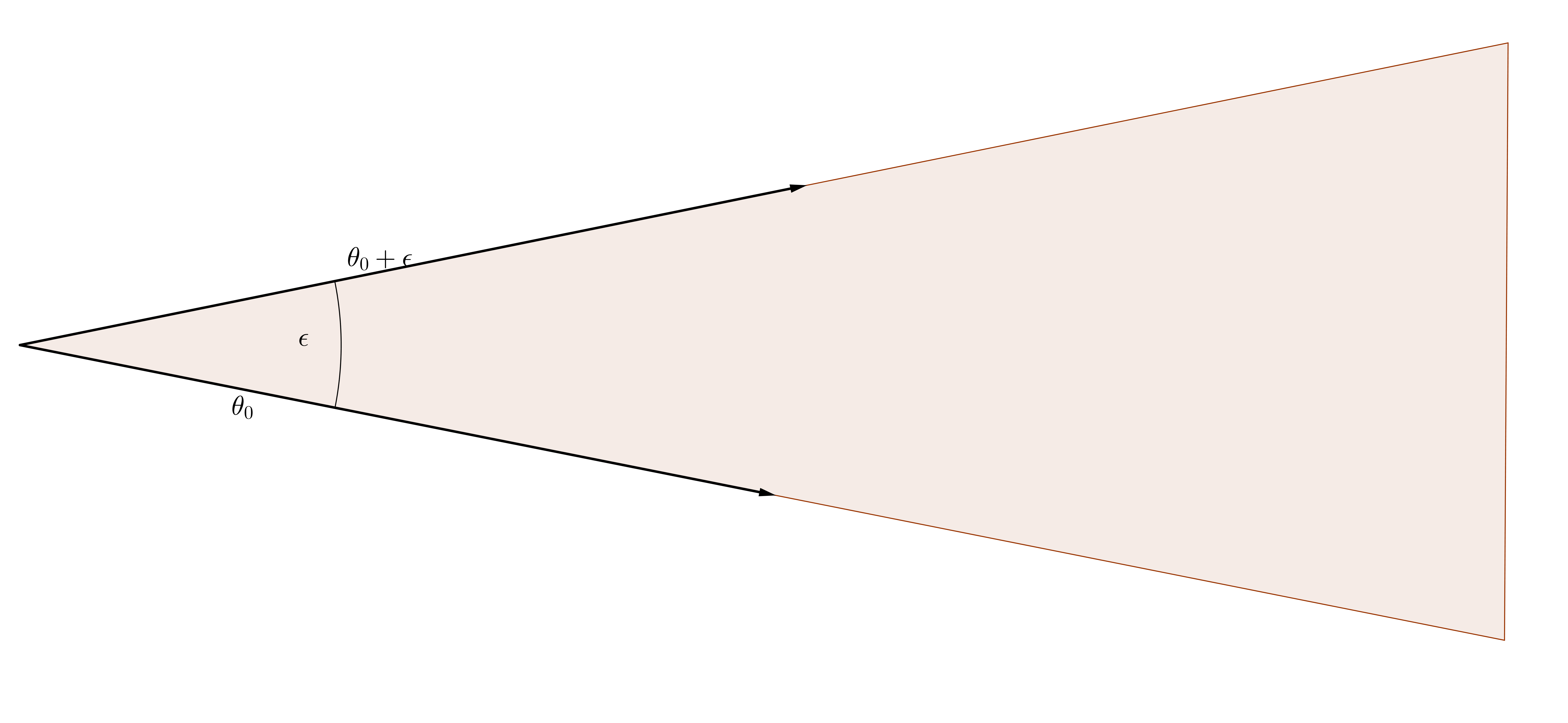}
 		\caption{a V-shaped region in $Q_{\infty}$}
 		\label{vshape}
 	\end{figure}
 	Let $k$ be the smallest integer such that $\omega_k\neq \omega'_k$. Consider the polygon $Q_{k-1}$. This polygon has at least one vertex $z \in V_0$ inside the region bounded by $\gamma_{k-1}$ and $\gamma'_{k-1}$. Let $S_{\delta}$ be a circle of radius $\delta$ with center $z$.
     
     From the Figure \ref{vshape2} it follows that each trajectory having direction $\theta \in [\theta_0+\theta',\theta_0+\theta'+\alpha]$ is not in $N_{\delta}$. But it is evident that there exists a constant $c >0$ such that
     
     \begin{align}
     \alpha\geq c \epsilon \delta
    \end{align}
    Therefore 
    
    \begin{align}
    \frac{\rho(N_{\delta}\cap[\theta_0,\theta_0+\epsilon])}{\epsilon} \leq \frac{\epsilon-c \epsilon \delta}{\epsilon}=1 - c \delta,
    \end{align}
    \noindent which contradicts with the assumption that $\theta_0$ is a point of density.
    
    If $q$ is a non-singular boundary point, then we can argue similarly. Of course, the distance between $V_0$ and  any trajectory originating from a singular point is zero.
     
\end{proof}

\begin{figure}
	\hspace*{-10cm}                                                           
	
	\centering
	\includegraphics[scale=0.6]{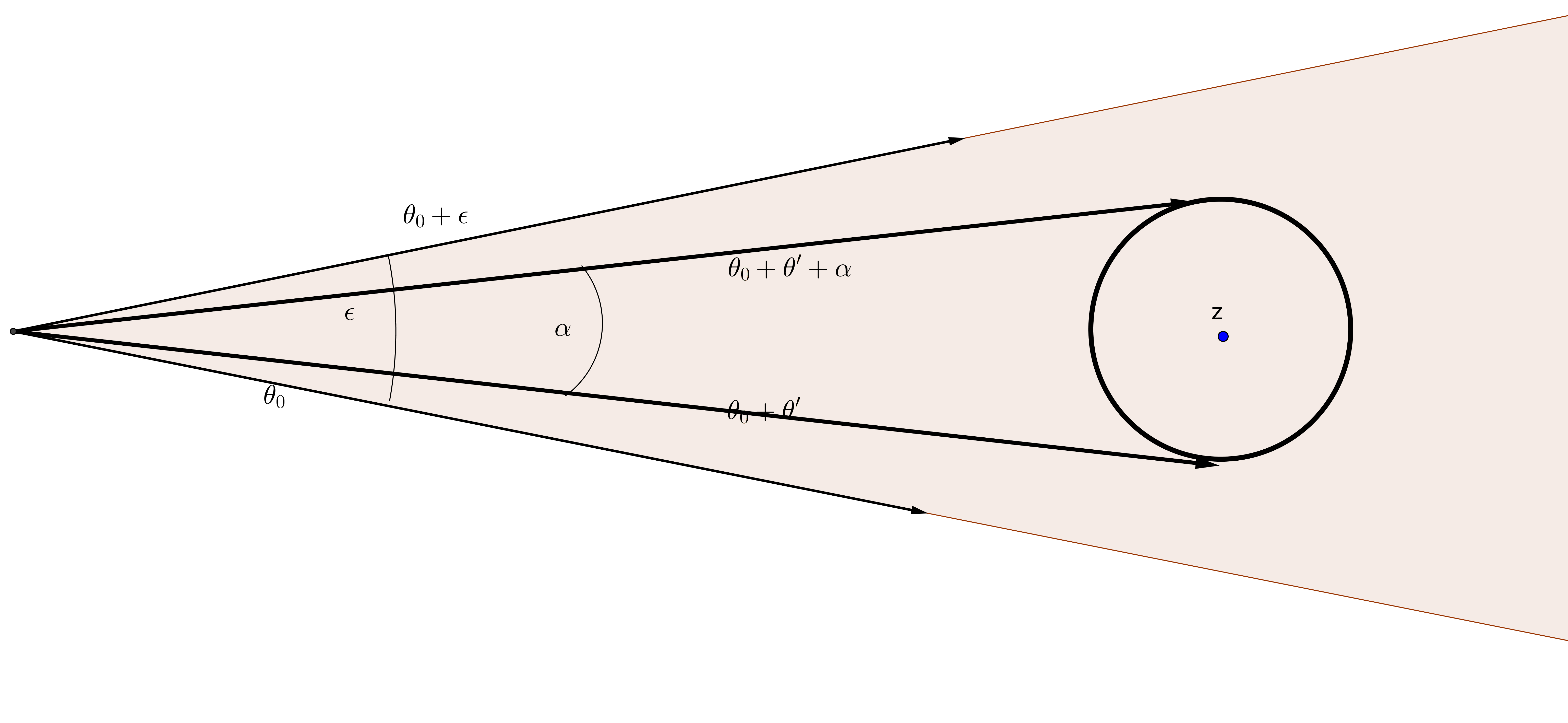}
	\caption{}
	\label{vshape2}
\end{figure}

 \section{Katok-Zemlyakov Construction for Flat Disks}
 \label{KZ}
 In this section we first recall Katok-Zemlyakov construction \cite{top-tran} for the polygons and then make a similar construction for the flat disks. We show that the translation surface  that we construct is unique; it does not depend on the way that we cut the flat disk. Also, we calculate the genus of this translation surface.
 
 \subsection{Polygons}
 \label{KZpolygons}
 Here we summarize how  we obtain a translation surface out of a rational polygon. Let $P \subset \R^2$ be a polygon. Let $y_1,\dots,y_n$ be the vertices and $e_1,\dots,e_n$ be the edges of $P$. Assume that the angle at $y_i$, $\theta_{y_i}$, is equal to $\pi\frac{k_i}{l_i}$, where $gcd(k_i,l_i)=1$. Let $l=lcm(l_1,\dots,l_n)$. Let $G$ denote the linear part of the group generated by reflections across the edges $e_1,\dots,e_n$. Then $G$ has $2l$ elements. Indeed, $G$ is isometric to the dihedral group of order $2l$. 
 
 For each $g\in G$, let $gP$ be a polygon such that
 
 \begin{enumerate}
 	\item 
 	$gP$ is obtained by applying a translation of $\R^2$ to $g.P$,
 	\item
 	$gP\cap g'P=\emptyset$ if $g\neq g'$. 
 \end{enumerate}
 
 \noindent Label the edges of $gP$ by  the symbols $ge_1, \dots,ge_n$. The reflection across $ge_i$ gives an element of $G$. Let $\tau_{ge_i}$ denote this element. Consider the polygon 
 $$(\tau_{ge_i}g)P.$$
\noindent Glue the edge $(\tau_{ge_i}g)e_i$ of  $(\tau_{ge_i}g)P$ with the edge $ge_i$ of $gP$ for each $g \in G \ \text{and}\ 1\leq i \leq n$. By that way we get a flat surface without boundary. $(\tau_{ge_i}g)e_i$ and $ge_i$ are glued by translations, therefore the surface that we get is a translation surface. 

\paragraph{Example:} We can obtain a flat torus through from a square in $\R^2$. We can glue four copies of the square to get the torus. See Figure \ref{torus}. 

Call the surface that we obtain by the above construction $S(P)$. Let $d(P)$ be the doubling of $P$: the surface that we obtain by gluing two copies of $P$ along their boundaries. Then $S(P)$  naturally gives a branched covering of $d(D)$ which is local isometry outside of the ramification points.

\begin{figure}
	\hspace*{-10cm}                                                           
	
	\centering
	\includegraphics[scale=0.5]{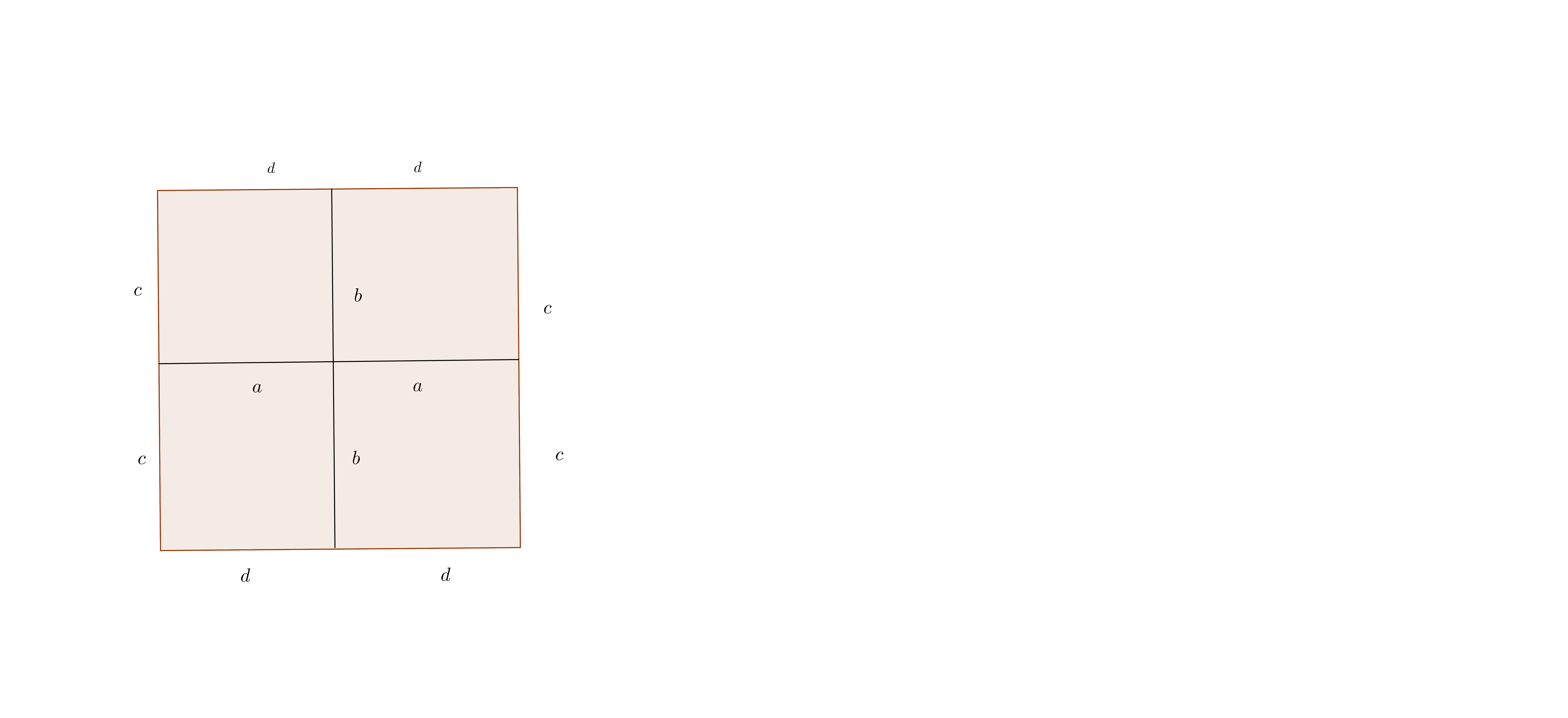}
	\caption{Glue the opposite edges of the big square to get a flat torus out of 4 copies of a square. }
	\label{torus}
\end{figure}

\subsection{Flat Disks} 
\label{KZflatdisks}
Now assume that we have a flat disk $D$. Let $x_1,\dots,x_m$ be the singular interior points of $D$. Assume that $\theta_{x_i}=2\pi \frac{k_i'}{l_i'}$ for each $i$, where $gcd(k_i',l_i')=1$. Let $e_1, \dots, e_n$ be the boundary components of $D$ and $y_1,\dots,y_n$ be the singular boundary points. Assume that $\theta_{y_j}=\pi \frac{k_j}{l_j}$ for each $j$, where $gcd(k_j,l_j)=1$ Let $a_1,\dots,a_m$ be the broken not self-intersecting geodesics such that

 \begin{enumerate}
	\item 
	$a_i$ joins $x_i$ and a singular boundary point $y_{j(i)}$,
	\item
	$a_i\cap b(D)=\{y_{j(i)}\}$,
	\item
	$a_i\cap a_j \subset b(D)$.
	
\end{enumerate}

Cut $D$ through $a_i$'s to get a flat disk $D^*$ which has $2m+n$ edges and a non-singular interior. For each $1\leq i \leq m$, there exist two boundary components of $D^*$ which are labeled by $a_i$. Since $D^*$ is oriented we can label one of them by $a_i^l$ and the other by $a_i^r$, where $l$ and $r$ refer to left and right, respectively. 

Since the interior of $D^*$ is non-singular there is a developing map

$$D^* \rightarrow \R^2$$

\noindent which is a local isometry outside of the singular points of $D^*$. Without loss of generality, 
we assume that this map is injective. That is, it is an isometry between $D^*$ and its image. Thus we assume that $D^*$ is a polygon in the plane.

Let $G$ be the linear part of the group of isometries of $\R^2$ 
 generated by the rotations of an angle $2\pi\frac{k_i'}{l_i'}$
about $x_i$, $1\leq i \leq m$, and the reflections across the edges $e_1,\dots,e_n$ of $D^*$. Let $l=lcm(l_1,\dots,l_n,l_1',\dots,l_m')$. Then $G$ has order $2l$. Indeed, $G$ is isometric to the dihedral group of order $2l$. 

For each $g \in G$ let $gD^*$ be a polygon obtained by a translation of  $g.D^*$ such that $gD^*\cap g'D^*=\emptyset$ whenever $g\neq g'$. Label the edges of $gD^*$ by the symbols 

$$ge_1,\dots,ge_n,ga_1^l,ga_1^r,\dots, ga_m^l,ga_m^r.$$

\noindent For each $ge_i$, let $\tau_{ge_i}$ be the linear part of the reflection across to $ge_i$. Glue $gD^*$ and $(\tau_{ge_i}g)D^*$ along the edges $ge_i$ and  $(\tau_{ge_i}g)e_i$. It is clear that we can do this by a translation. 

Let $r_{\theta}$ be  the counterclockwise rotation of angle $\theta$ about the origin. Glue the polygons $gD^*$ and $(r_{-\theta_{x_i}}g)D^*$ through the edges  $ga_i^r$ and $r_{-\theta_{x_i}}ga_i^l$. Note that this gluing can be given by a translation of $\R^2$. By that way we obtain a translation surface $S(D)$. Clearly, there is a branched covering $\Pi:S(D)\rightarrow d(D)$ which is a local isometry except at the ramification points. We call $S(D)$ invariant surface of $D$.

\begin{figure}
	\hspace*{-10cm}                                                           
	
	\centering
	\includegraphics[scale=0.5]{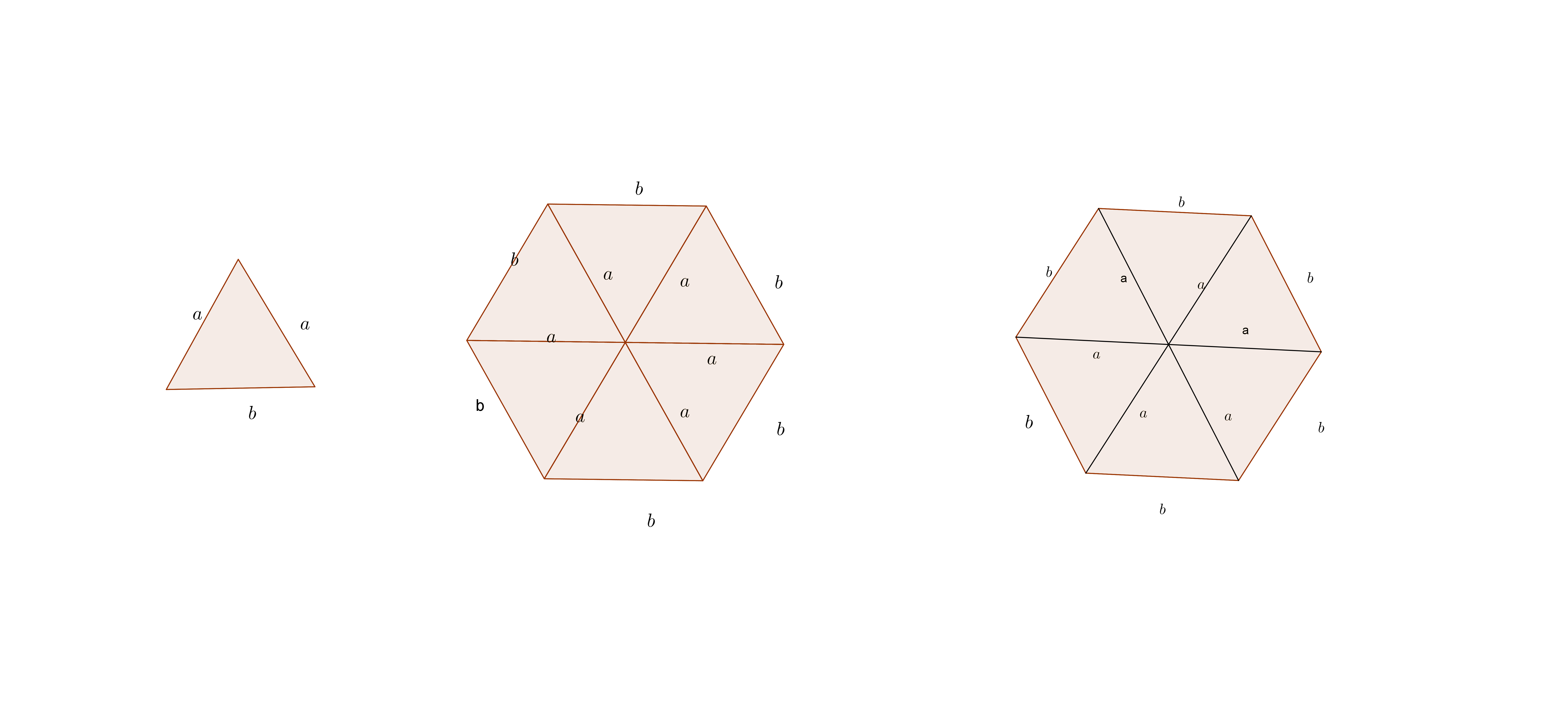}
	\caption{The translation surface of genus two from the flat disk obtained from an equilateral triangle }
	\label{hexagon}
\end{figure}

\paragraph{Example:}
Assume that $D$ is the flat disk obtained by gluing two edges of an equilateral triangle. Thus $D$ is a flat disk with one singular interior point and one singular boundary point with angles $\frac{\pi}{3}$ and $\frac{2\pi}{3}$. See the Figure \ref{hexagon}. Then $S(D)$ is obtained by gluing each edge of the hexagon on the middle by the opposite edge of the hexagon on the right. $S(D)$ has genus two.

\begin{figure}
	\hspace*{-10cm}                                                           
	
	\centering
	\includegraphics[scale=0.5]{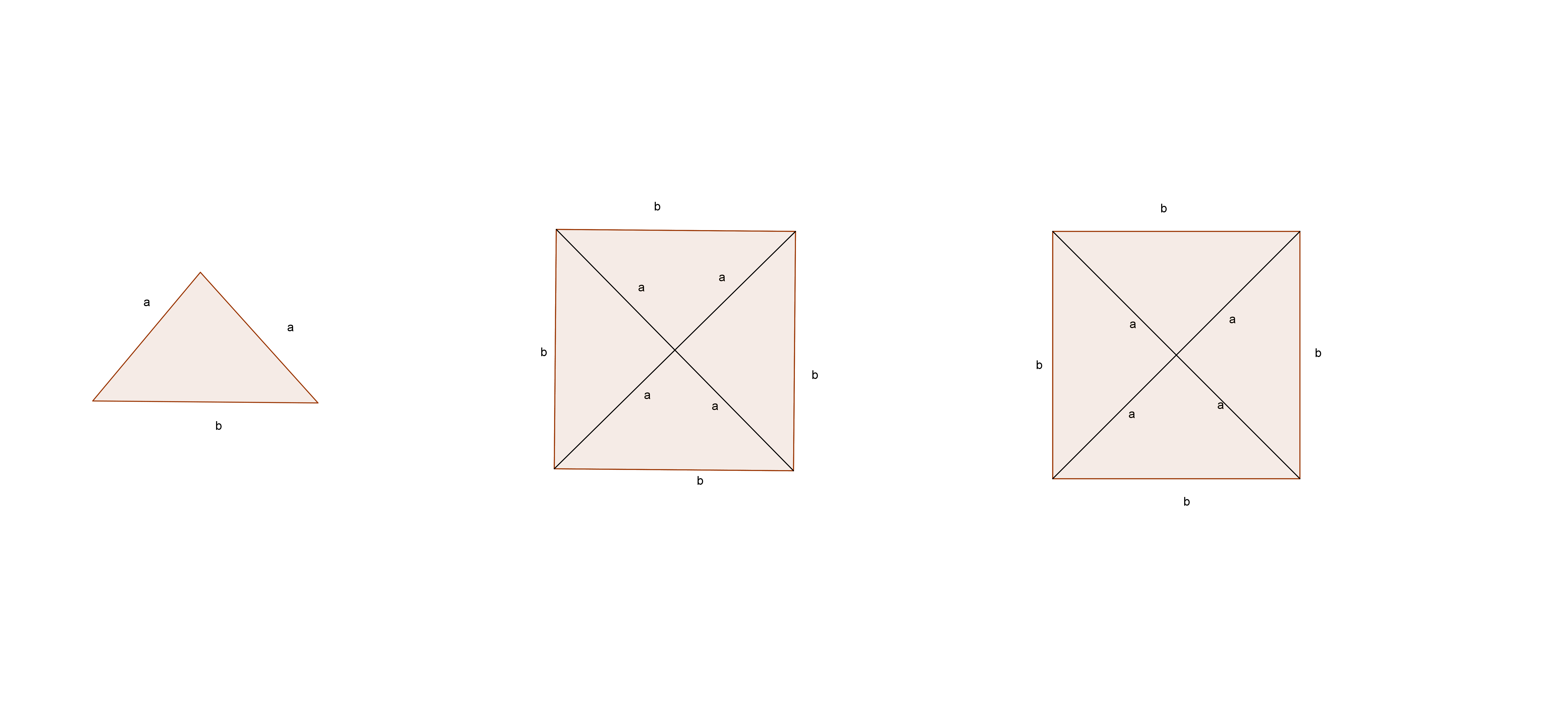}
	\caption{The flat torus which is the translation surface of  the flat disk obtained from a right isosceles triangle.  }
	\label{torus2}
\end{figure}
\paragraph{Example:}
Assume that $D$ is the flat disk obtained by gluing two edges of an isosceles right triangle. Thus $D$ is a flat disk with one singular interior point and one singular boundary point with angles $\frac{\pi}{4}$ and $\frac{\pi}{2}$. See the Figure \ref{torus2}. Then $S(D)$ is obtained by gluing each edge of the square on the middle by the opposite edge of the square on the right. $S(D)$ has genus one.

\begin{remark}
	If the boundary of $D$ is non-singular we can take any point $y$ in $b(D)$ and some broken geodesics $a_1\dots,a_m$ joining the singular points and $y$. Repeating the construction above, we get a translation surface $S(D)$ and a cover $\Pi:S(D)\rightarrow d(D)$. 
\end{remark}
\subsection{Uniqueness of Invariant Surface}
 
\label{uniqueness}

In this section we prove that the surface $S(D)$ which was found in the previous section does not depend on the chosen paths $a_1,\dots,a_m$. 
First we introduce a notion of covering for the flat surfaces. Then we state some results found in \cite{ISH}. From a topological point of view such a covering is nothing else than a branched covering. But, we also require that it respects the flat metrics. 

Given a branched covering $\psi : S^* \rightarrow S$ between two surfaces, we denote the set of the branched points by $\frak{b}$ and the set of the ramification points by $\frak{r}$. 

Let $S^*$ and $S$ be two flat surfaces.
\begin{udefinition}
	A map $\psi:\ S^* \rightarrow S $ is called a flat (covering) map if it is a
	branched covering and
	$$\psi\restriction_{S^*-\psi^{-1}(\frak{b})}: \  S^*-\psi^{-1}(\frak{b}) \rightarrow S-\frak{b}$$
	is a local isometry.
\end{udefinition}

\paragraph{Examples:}
\begin{enumerate}
	\item
	If $G$ is a finite group of isometries of $S^*$ then $S^*\rightarrow S^*/G$ is a flat covering.
	
	\item
	
	The map $\Pi: S(D)\rightarrow d(D)$ obtained in Section \ref{KZflatdisks} is flat.
	
	\item If $\phi: S'\rightarrow S$ is a branched covering then the induced map on $S'$ is flat, and this makes $\phi$ a flat covering.
	\end{enumerate}

\begin{udefinition}
	\begin{enumerate}
		\item 
		A closed, orientable flat surface is called really flat if it has finite holonomy group.
		
		\item
		
		An orientable compact flat surface with boundary is called really flat if its doubling is really flat.
		
		\item
		A non- orientable compact surface is really flat if its orientable double cover is really flat.
	\end{enumerate}
\end{udefinition}

Thus the really flat surfaces are somewhere in between the translation surfaces and arbitrary flat surfaces. Note that a flat disk is really flat if and only if the angles at singular points are rational multiples of $\pi$. 

\begin{theorem}
	Let $S$ be a compact, oriented, closed really flat surface.
	\begin{enumerate}
		\item 
		There exists a translation surface $S^*$ and a flat covering $\psi: S^*\rightarrow S$. Also, this map corresponds to the kernel of the holonomy represntation of $S$. 
		
		\item
		$\psi$ is a cyclic Galois covering.
		
		\item
		If $x \in S$ has angle $2\pi\frac{k}{l}$, $gcd(k,l)=1$, then the ramification index at each point in the fiber of $x$ is $l$.
		\item
		The degree of $\psi$ is the order of the holonomy group.
	\end{enumerate} 
\end{theorem}

We call $\psi: S^* \rightarrow S$ canonical cover. Now let $D$ be a rational flat sphere. As before, let $x_1\dots x_m$ be the singular interior points and $y_1,\dots,y_n$ be the singular boundary points. Assume that the angle at $x_i$ is $2\pi \frac{k_i'}{l_i'}$ and the angle at $y_j$ is $\pi \frac{k_j}{l_j}$, where $k_i'$ and $l_i'$ are coprime integers, and so are $k_j$ and $l_j$. Let $\Pi^*:d(D)^*\rightarrow d(D)$ be the canonical covering of $d(D)$. 

\begin{theorem}
	
Let $l$ be the leaast common multiple of $l_1,\dots,l_n,l_1',\dots,l_m'$. Then

\begin{enumerate}
	\item 
	the degree of $\Pi^*$ is $l$,
	\item
	the ramification index at each point in the fiber of $x_i$ or its mirror image is $l'_i$,
	\item
	the ramification index of each point in the fiber of $y_j$ is $l_j$,

\end{enumerate}
\end{theorem}
\noindent Now we prove that the surfaces $S(D)$ and $d(D)^*$ are same.

\begin{theorem}
	
	There is an isometry $\phi: S(D) \rightarrow d(D)^*$ making the following diagram commutative:
	
	\[
	\begin{tikzcd}
	S(D)\arrow{r}{\phi} \arrow{rrd}{\Pi}
	&	d(D)^*\arrow{rd}{\Pi^*}
	\\
	&&
d(D)
	\end{tikzcd}
	\]
	\begin{proof}
		$S(D)$ has trivial holonomy group, so it corresponds to a subgroup of the kernel of the holonomy representation of $d(D)^*$. But, the degree of $\Pi: S(D) \rightarrow d(D)$ is $l$, which is the degree of the holomomy group. Thus $S(D)$ corresponds to the kernel of the holonomy group. But, $d(D)^*$ also corresponds to the kernel of the holonomy group. Therefore $\Pi: S(D)\rightarrow d(D)$ and $\Pi^*: d(D)^*\rightarrow d(D)$ are same as flat coverings.
	\end{proof}
\end{theorem}

 \subsection{Calculation of Euler Characteristics}
 \label{eulerchar}
 We first give Riemann-Hurwitz formula. Let $S''$ and $S'$ be two compact orientable surfaces and $\psi: S'' \rightarrow S'$ be a branched covering. The following formula is called Riemann-Hurwitz formula.
 
 \begin{align}
 \chi(S'')=N\chi(S')-deg(R),
 \end{align}
 
 \noindent 
 where $N$ is the degree of $\psi$ and $R$ is the ramification divisor:
 $$R=\sum_{p\in S''}(e_p-1)p.$$
 Here $e_p$ is the ramification index of $\psi$ at $p$.

 As before, let $D$ be a flat disk with $m$ singular interior points $x_1,\dots,x_m$ and $n$ singular boundary points $y_1,\dots,y_n$. Let $\theta_{x_i}=2\pi\frac{k_i'}{l_i'}$ and $\theta_{y_j}=\pi \frac{k_j}{l_j}$. Consider the cover $\Pi^*:d(D)^*\rightarrow d(D)$.
 
 \begin{enumerate}
 	\item 
 	It has degree $l$,  where $l$ is the least common multiple of $l_1',\dots l_m',l_1,\dots,l_n$.
 	\item
 	The ramification index at a point $x_i$ or its mirror image is $l_i'$.
 	\item
 	The ramification index at a point over $y_j$ is $l_j$.
 	\item
 	The set ${\Pi^*}^{-1}(x_i)$ consists of $\frac{l}{l_i'}$ points. 
 	\item
 	If $x'_i$ is the mirror image of $x_i$, then the set ${\Pi^*}^{-1}(x_i')$ consists of $\frac{l}{l_i'}$ points. 
 	\item
 	The set ${\Pi^*}^{-1}(y_j)$ consists of $\frac{l}{l_j}$ points. 
  \end{enumerate} 
  
  It follows that \begin{align}
  deg (R)= 2\sum_{i=1}^m \frac{l}{l_i'}(l_i'-1)+\sum_{j=1}^{n}\frac{l}{l_j}(l_j-1).
  \end{align}
 
 Therefore 
 
 \begin{align}
 \chi(S(D))= 2l - 2\sum_{i=1}^m \frac{l}{l_i'}(l_i'-1)-\sum_{i=1}^{n}\frac{l}{l_j}(l_j-1).
 \end{align}

\end{document}